\documentclass[11pt,leqno]{article}

\usepackage[active]{srcltx}
\usepackage{amsfonts,latexsym,amsmath,amssymb,amsthm}
\usepackage{dsfont}
\usepackage{fullpage}
\newtheorem{theorem}{Theorem}[section]
\newtheorem{lemma}[theorem]{Lemma}

\newtheorem{proposition}[theorem]{Proposition}
\newtheorem{corollary}[theorem]{Corollary}
\newtheorem{remark}[theorem]{Remark}

\theoremstyle{remark}

\newtheorem*{note*}{Note}

\numberwithin{equation}{section}

\makeatletter
\newcommand{\vol}{\mathrm{vol}}
\newcommand{\ls}{\leqslant}
\newcommand{\gr}{\geqslant}
\makeatother

\usepackage{hyperref}

\begin{document}
\small

\title{\bf On the maximal perimeter of isotropic log-concave probability measures}

\author{Silouanos Brazitikos, Apostolos Giannopoulos, Antonios Hmadi\\ and Natalia Tziotziou}

\date{}

\maketitle

\begin{abstract}\footnotesize 
We study the maximal perimeter constant of isotropic log-concave probability measures on $\mathbb{R}^n$.
For a measure $\mu$, this quantity, denoted by $\Gamma(\mu)$, is defined as the supremum of the $\mu$-perimeter over all convex bodies and measures the largest possible boundary contribution of convex sets with respect to $\mu$.
Let
$$\Gamma_n := \sup\{\Gamma(\mu) : \mu \text{ is an isotropic log-concave probability measure on } \mathbb{R}^n\}.$$
We prove that $\Gamma_n \ls Cn^{3/2}$, where $C>0$ is an absolute constant. 
This result improves the previously known $O(n^2)$ upper bound. 
Under additional structural assumptions, we obtain sharp linear bounds of order $O(n)$.
\end{abstract}

\section{Introduction}\label{section:1}

We study extremal perimeter properties of high-dimensional log-concave probability measures. 
Our main result establishes an $O(n^{3/2})$ upper bound for the maximal perimeter constant of isotropic log-concave probability measures on $\mathbb{R}^n$, improving the previously best known quadratic bound.

For a probability measure $\mu$ on $\mathbb{R}^n$ and a convex body $A \subset \mathbb{R}^n$, the $\mu$--perimeter of $A$ is defined by
$$\mu^+(\partial A)=\liminf_{\epsilon \to 0^+}\frac{\mu\big((A+\epsilon B_2^n)\setminus A\big)}{\epsilon},$$
where $B_2^n$ denotes the Euclidean unit ball.
This notion coincides with the Minkowski content of $\partial A$ with respect to $\mu$ and quantifies the first-order variation of the $\mu$-measure of $A$ under infinitesimal Euclidean enlargements.
The maximal perimeter constant of $\mu$ is then defined as 
$$\Gamma(\mu)=\sup\big\{\mu^+(\partial A):\ A \subset \mathbb{R}^n \text{ is a convex body}\big\},$$
and measures the largest possible boundary contribution among all convex bodies with respect to $\mu$.

\medskip 

\noindent {\bf \S 1.1. Main results.} 
A probability measure $\mu$ on $\mathbb{R}^n$ is log-concave if it admits a density $f$ satisfying
$$\ln f(\tau x + (1-\tau)y) \gr \tau \ln f(x)+(1-\tau)\ln f(y)\quad \text{for all } x,y\in\mathbb{R}^n, \ \tau\in[0,1].$$
The measure $\mu$ is called isotropic if its barycenter is at the origin and its covariance matrix is the identity.  

In this work, we investigate the growth rate of
$$\Gamma_n:=\sup\{\Gamma(\mu):\mu\;\text{is an isotropic log-concave probability measure on}\;\mathbb{R}^n\}.$$
A basic example is provided by the uniform measure on the cube $Q_n$ in isotropic position, for which $\Gamma(\mu_{Q_n})\approx n$, where $\mu_{Q_n}$ is the 
associated isotropic measure. More generally, in Section~\ref{section:5}  we show that for any isotropic convex body $K\subset\mathbb{R}^n$, the maximal perimeter constant of the uniform measure on $K$ is of the same order as its surface area $S(K)$ (see Section~\S 5.1). 
In particular, when $K = Q_n$, this quantity is of order $n$.

Prior to the present work, the best known general upper bound was quadratic, namely $\Gamma_n \ls C n^2$, due to Livshyts \cite{Livshyts-2021}.
Our main result improves the bound to order $n^{3/2}$ in full generality.

\begin{theorem}\label{th:main-upper-bound}
Let $\mu$ be an isotropic log-concave probability measure on $\mathbb{R}^n$. 
Then
$$\Gamma(\mu)\ls Cn^{3/2},$$
where $C>0$ is an absolute constant.
\end{theorem}

Together with the example of the cube, Theorem~\ref{th:main-upper-bound} yields
$$cn\ls \Gamma_n \ls Cn^{3/2}.$$

The estimate of Theorem~\ref{th:main-upper-bound} can be improved if one restricts attention to the symmetric maximal perimeter constant of $\mu$, defined by
$$\Gamma^{(s)}(\mu)=\sup\big\{\mu^+(\partial A):\ A \subset \mathbb{R}^n \text{ is a symmetric convex body}\big\}.$$

\begin{theorem}\label{th:symmetric-main-upper-bound}
Let $\Gamma_n^{(s)}:=\sup\{\Gamma^{(s)}(\mu):\mu\;\text{is an isotropic log-concave probability measure on}\;\mathbb{R}^n\}$.
Then
$$\Gamma_n^{(s)}\ls 4n.$$
\end{theorem}

In Section~\ref{section:5} we further show that better bounds can be obtained under additional structural assumptions. 
Specifically, we prove linear $O(n)$ bounds for several important classes of isotropic geometric log-concave measures, 
including uniform measures on convex bodies, $1$-symmetric measures and measures with homogeneous level sets. These cases are already mentioned
in \cite{Livshyts-2021}. Our approach allows us to obtain sharper bounds in some of these cases and applies to more general classes:

\smallskip 

(i) In Proposition~\ref{prop:main-estimate-body-case} we show that if $K$ is an isotropic convex body in $\mathbb{R}^n$ and $\mu_K$ denotes 
the associated isotropic probability measure with density $L_K^n\mathds{1}_{K/L_K}$ where $L_K$ is the isotropic constant of $K$, then
$$\Gamma(\mu_K)=L_KS(K),$$
where $S(K)$ is the surface area of $K$. As a consequence, $\Gamma(\mu_K)\ls \sqrt{\frac{n}{n+2}}\,n$. This inequality is sharp: we check 
that equality holds if $K$ is an isotropic regular simplex.

\smallskip 

(ii) In Theorem~\ref{th:unconditional} we prove that if $\mu$ is a $1$-unconditional isotropic log-concave probability measure,
meaning that its density $f$ satisfies $f(x_1,\ldots,x_n)=f(\epsilon_1 x_1,\ldots,\epsilon_n x_n)$ for all choices of signs $\epsilon_i\in\{-1,1\}$, then
$$\Gamma(\mu)\ls\sqrt{2}n.$$
The proof is independent of the methods in \cite{Livshyts-2021} and covers the case of $1$-symmetric measures which is strictly smaller.

\smallskip 

(iii) In Theorem~\ref{thm:1d_isotropic_sharp} we show that if $\mu$ is an isotropic log-concave probability measure on $\mathbb{R}$, then
$$\Gamma(\mu)\ls 2.$$
Moreover, the constant $2$ is sharp: there exist isotropic log-concave probability measures $\mu$ with
$\Gamma(\mu)$ arbitrarily close to $2$, with one-sided exponentials being extremal.

\smallskip 

(iv) In Theorem~\ref{thm:product-sharp} we prove that if $\mu=\mu_1\otimes\cdots\otimes\mu_n$ is a product probability measure on $\mathbb{R}^n$,
where each $\mu_k$ has a density $g_k\in L^\infty(\mathbb{R})$, then
$$\Gamma(\mu)\ \ls\ 2\sum_{k=1}^n \|g_k\|_\infty.$$
Moreover, the constant $2$ is optimal. In particular, if each $\mu_k$ is an isotropic log-concave probability measure on $\mathbb{R}$, then 
$$\Gamma(\mu)\ls 2n.$$

\medskip 

It would be interesting to determine whether a linear bound $O(n)$ holds in full generality.

\medskip 

\noindent {\bf \S 1.2. Elements of the proof.} 
The proof of Theorem~\ref{th:symmetric-main-upper-bound} relies on two elementary yet robust ingredients. 
First, a dilation inequality for centered log-concave measures provides an upper bound on the perimeter of a convex body $A$ 
containing the origin in terms of its measure and its inradius. 
Second, for isotropic log-concave measures, the measure of any convex body admits a one-dimensional upper bound in terms of its minimal width. 
In the symmetric setting, the inball of $A$ is necessarily centered at the origin and the minimal width equals twice the inradius, which
immediately yields a linear bound for the symmetric maximal perimeter constant.

The proof of the general bound in Theorem~\ref{th:main-upper-bound} follows the same overarching strategy, supplemented with ideas introduced in \cite{Livshyts-2021}. 
A central role is played by the super-level sets 
$$R_t(\mu)=\{x\in\mathbb{R}^n: f(x)\gr e^{-t}\|f\|_{\infty}\}$$ 
of the density $f$ of $\mu$, which form an increasing family of bounded convex sets. 
For isotropic log-concave measures, these sets exhaust the measure rapidly: for $t$ of order $n$, they already carry almost full mass, 
while their inradius is bounded from below by an absolute constant. 
This structural information allows one to localize the perimeter problem to regions where the density is well controlled.

The starting point is a perimeter estimate for an arbitrary convex body $A$ in terms of its inradius and the ratio between the maximum of the density 
on $A$ and its value at the center of an inball. 
Combining this estimate with Steinhagen's inequality (Theorem~\ref{th:steinhagen}), which compares minimal width and inradius up to a factor $\sqrt{n}$, 
yields a bound of the form $$\sqrt{n}\left(n+\ln\left(\|f\big|_A\|_{\infty}/f(x_A)\right)\right),$$ where $x_A$ denotes the center of an inball of $A$.
When $A$ is contained in a fixed level set $R_{6n}(\mu)$, the logarithmic term is of order at most $n$, leading to an $O(n^{3/2})$ estimate.
To treat arbitrary convex bodies, the boundary is decomposed according to density levels: the contribution coming from outside $R_{6n}(\mu)$ 
is controlled by comparison with the boundaries of the level sets themselves.
Using the structural properties of $R_t(\mu)$, this contribution is shown to be linear in $n$.
Combining the two regimes yields the $O(n^{3/2})$ upper bound in full generality.

\medskip 

\noindent {\bf \S 1.3. Background and related work.} 
The Gaussian maximal perimeter constant arises naturally in several areas of high-dimensional probability, convex geometry and theoretical computer science.
Kwapie\'{n} and Mushtari posed the problem of determining the order of growth of $\Gamma(\mu)$ when $\mu=\gamma_n$ is the standard Gaussian measure on $\mathbb{R}^n$.
Simple arguments (see \cite{Ball-1993}) yield the bounds $\sqrt{\ln n}\lesssim \Gamma(\gamma_n)\lesssim \sqrt{n}$ with absolute implicit constants.
A breakthrough result of Ball \cite{Ball-1993} showed that
$$\Gamma(\gamma_n)\ls 4n^{1/4}.$$
This estimate was later shown to be sharp by Nazarov \cite{Nazarov-2003}, who proved that for every $n\times n$ symmetric positive definite matrix $T$, the Gaussian measure with density proportional to $\exp(-\langle Tx,x\rangle/2)$ has maximal perimeter of order $\sqrt{\|T\|_{\mathrm{HS}}}$, where $\|\cdot\|_{\mathrm{HS}}$ denotes Hilbert-Schmidt norm. 
Since $\|I_n\|_{{\rm HS}}=n^{1/2}$ for the identity matrix $I_n$, it follows that
\begin{equation}\label{eq:gaussian}
\Gamma(\gamma_n)\approx n^{1/4}.
\end{equation}
These results highlight a non-classical scaling behavior of Gaussian surface area in high dimensions, in contrast to the
linear growth observed for the uniform measure on the cube. They also motivate the investigation of analogous bounds beyond the Gaussian framework. 
A natural and substantially broader class of measures is given by log-concave probability measures, which arise ubiquitously in convex geometry and probability theory and include uniform measures on convex bodies as well as many exponential-type distributions.  

In a series of works, Livshyts studied extensions of Gaussian perimeter bounds to the log-concave setting.
In particular, for the probability measure $\mu_{p,n}$ with density proportional to $\exp(-|x|^p/p)$, $0<p<\infty$, it was shown in \cite{Livshyts-2013} that
$$\Gamma(\mu_{p,n})\approx C_p\,n^{\frac{3}{4}-\frac{1}{p}},$$
interpolating between Gaussian and uniform behavior. 
Subsequently, almost sharp upper bounds for general rotationally invariant log-concave measures were obtained in terms of the expectation and variance of the Euclidean norm in \cite{Livshyts-2014}, while maximal perimeter estimates for convex polytopes, with logarithmic dependence on the number of facets, were studied in \cite{Livshyts-2015}.

More recently, Livshyts \cite{Livshyts-2021} derived general lower bounds for $\Gamma(\mu)$ under mild regularity assumptions.
Specifically, if $\mu$ is absolutely continuous and satisfies
$\sqrt{\mathrm{Var}_{\mu}(|x|)}\ls \alpha\,\mathbb{E}_{\mu}(|x|)$ for some $\alpha\in [0,1)$,
then
\begin{equation}\label{eq:galyna-lower-bound}
\Gamma(\mu)\gr C(\alpha)\,
\frac{\sqrt{n}}{\sqrt[4]{\mathrm{Var}_{\mu}(|x|)}\,\sqrt{\mathbb{E}_{\mu}(|x|)}}.
\end{equation}
If $\mu$ is isotropic, then $\mathbb{E}_{\mu}(|x|)\approx\sqrt{n}$ and $\mathrm{Var}_{\mu}(|x|)\ls C$ for an absolute constant $C$, the latter estimate following from the recent resolution of the thin-shell conjecture by Klartag and Lehec \cite{Klartag-Lehec-2025b}.
Consequently,
\begin{equation}\label{eq:lower-bound-isotropic}
\Gamma(\mu)\gtrsim n^{1/4}
\end{equation}
for every isotropic log-concave probability measure on $\mathbb{R}^n$, up to an absolute constant. In view of \eqref{eq:gaussian}, this shows that within this class the Gaussian measure is essentially extremal on the lower end. 

Livshyts also proved in \cite{Livshyts-2021} that for every log-concave probability measure $\mu$ on $\mathbb{R}^n$ with density $f$, 
there exists a volume preserving transformation $T\in SL_n$ such that the push-forward measure $T_{\ast}\mu$ satisfies
\begin{equation}\label{eq:almost-n}
\Gamma(T_{\ast}\mu)\ls Cn\|f_T\|_{\infty}^{1/n}=Cn\|f\|_{\infty}^{1/n}
\end{equation}
for an absolute constant $C>0$, where $f_T(x)=f(T^{-1}(x))$ denotes the density of $T_{\ast}(\mu)$ and $\|f_T\|_{\infty}=\|f\|_{\infty}$. 
Note that $\Gamma(\mu)$ is not invariant under volume preserving transformations.

If $\mu$ is an isotropic log-concave probability measure, then $\|f\|_{\infty}^{1/n}$ coincides with the isotropic constant of $f$, 
which is bounded by an absolute constant (see Section~\ref{section:2}). 
Therefore, \eqref{eq:almost-n} implies that every isotropic log-concave probability measure $\mu$ on $\mathbb{R}^n$ admits 
a push-forward $T_{\ast}\mu$ such that 
$$\Gamma(T_{\ast}\mu)\ls Cn.$$ 
This result once again raises the natural question of whether $\Gamma_n$ grows linearly with the dimension.

\section{Notation and preliminaries}\label{section:2}

We work in $\mathbb{R}^n$, equipped with the standard inner product $\langle \cdot, \cdot \rangle$.  
The associated Euclidean norm is denoted by $|\cdot|$, the Euclidean unit ball by $B_2^n$, and the Euclidean unit sphere by $S^{n-1}$.  
Lebesgue measure in $\mathbb{R}^n$ is denoted by $\vol_n$, and we write $\omega_n = \vol_n(B_2^n)$ for the volume of the Euclidean unit ball.
We denote by $\sigma$ the rotationally invariant probability measure on $S^{n-1}$.

Throughout the paper, the symbols $C, c, c', c_1, c_2, \ldots$ denote absolute positive constants whose values may change from line to line.  
Whenever we write $a \approx b$, we mean that there exist absolute constants $c_1, c_2 > 0$ such that $c_1 a \ls b \ls c_2 a$. 
All absolute constants are independent of the dimension $n$.

\smallskip 

\noindent {\bf \S 2.1. Convex bodies.} 
A convex body in $\mathbb{R}^n$ is a compact convex set $K$ with nonempty interior.  
It is called symmetric if $K = -K$, and centered if its barycenter $\operatorname{bar}(K) = \frac{1}{\vol_n(K)} \int_K x\,dx$ is at the origin. 
For every convex body $K\subset \mathbb R^n$ we denote by $\overline{K}$ the homothetic copy of $K$ scaled to have unit volume,
namely $\overline{K}:=\vol_n(K)^{-1/n}K$.

Let $K$ be a convex body in $\mathbb{R}^n$ with $0\in {\rm int}(K)$. 
The radial function of $K$ is defined by $\varrho_K(x) = \max\{ t > 0 : t x \in K \}$ for all nonzero $x$, and the support function of $K$ is given by 
$h_K(x)=\max\{\langle x,y\rangle :y\in K \}$ for all $x\in\mathbb{R}^n$. 

If a probability measure $\mu$ on $\mathbb{R}^n$ admits a density $f$ with respect to Lebesgue measure, then
$$\mu^+(\partial K)=\int_{\partial K}f(x)\,d\lambda(x),$$
where $\lambda $ denotes the $(n-1)$-dimensional Hausdorff measure on $\partial K$. 
The surface area $S(K)$ of $K$ is defined by
$$S(K):=\lambda (\partial K).$$
A well-known inequality that will be useful in our study asserts that if $0\in {\rm int}(K)$ then 
\begin{equation}\label{eq:surface-inradius}
S(K)\ls\frac{n\vol_n(K)}{r(K)},
\end{equation}
where $r(K)$ is the inradius of $K$ with respect to the origin, that is, the largest $r>0$ such that $rB_2^n\subseteq K$.
For completeness, we sketch a proof. Setting $r=r(K)$, we write
\begin{align*}
S(K) &=\liminf_{\epsilon\to 0}\frac{\vol_n(K+\epsilon B_2^n)-\vol_n(K)}{\epsilon}\ls \liminf_{\epsilon\to 0}\frac{\vol_n(K+(\epsilon /r)K)-\vol_n(K)}{\epsilon}\\
&=\lim_{\epsilon\to 0}\frac{(1+\epsilon/r)^n-1}{\epsilon }\vol_n(K)=\frac{n\vol_n(K)}{r}.
\end{align*}
The (intrinsic) inradius of $K$ is defined by
$$r_K=\sup\{r>0:\exists x\in K\;\text{such that}\;x+rB_2^n\subseteq K\}.$$
The existence of an inball attaining this supremum follows from Blaschke's selection theorem. In general, the inball
need not be unique, as shown by the example of a circular cylinder.

Another parameter that will play an important role is the minimal width of $K$. The width of $K$ in the direction of $\xi\in S^{n-1}$ is defined by
$$w_K(\xi)=h_K(\xi)+h_K(-\xi)$$
and the minimal width of $K$ is given by 
$$w_K=\min\{w_K(\xi):\xi\in S^{n-1}\}.$$
Equivalently, $w_K$ is the minimal distance between two parallel supporting hyperplanes of $K$.
 
A convex body $K$ in ${\mathbb R}^n$ is called isotropic if it has volume $1$, is centered, and its covariance matrix is a multiple of the identity. 
Equivalently, there exists a constant $L_K>0$, called the isotropic constant of $K$, such that
\begin{equation*}
\|\langle \cdot ,\xi\rangle\|_{L_2(K)}^2:=\int_K\langle x,\xi\rangle^2dx =L_K^2\quad\text{for all}\;\xi\in S^{n-1}.
\end{equation*}
We shall use several geometric properties of isotropic convex bodies. 
First, it was shown in~\cite{Kannan-Lovasz-Simonovits-1995} that if $K$ is isotropic, then
\begin{equation}\label{eq:KLS-radius}
\sqrt{\frac{n+2}{n}}L_KB_2^n \subseteq K \subseteq \sqrt{n(n+2)}L_KB_2^n.
\end{equation}
Moreover, it is known (see \cite[Proposition~3.3.1]{BGVV-book}) that $L_K\gr L_{B_2^n}\gr c$ for an absolute constant $c>0$. 

Bourgain's slicing problem \cite{Bourgain-1986} asks whether there exists an absolute constant $C>0$ such that
\begin{equation}\label{eq:conjecture}
L_n:=\max\{ L_K:K\ \hbox{is an isotropic convex body in}\ \mathbb{R}^n\}\ls C.
\end{equation}
An affirmative solution was recently obtained by Klartag and Lehec~\cite{Klartag-Lehec-2025}, following an important
contribution by Guan~\cite{Guan-preprint} (see also~\cite{Bizeul-2025} for an alternative proof). 
Consequently, $L_K\approx 1$, uniformly in $n$, for every isotropic convex body $K$ in $\mathbb{R}^n$.
For further background, we refer to the survey~\cite{Giannopoulos-Pafis-Tziotziou-2025}.

\medskip 

\noindent {\bf \S 2.2. Log-concave probability measures.} 
A Borel measure $\mu$ on $\mathbb R^n$ is called log-concave if 
$$\mu(\tau A+(1-\tau)B) \gr \mu(A)^{\tau}\mu(B)^{1-\tau}$$ for any compact sets $A$
and $B$ in ${\mathbb R}^n$, and any $\tau \in (0,1)$. 
Borell \cite{Borell-1974} proved that if a log-concave probability measure $\mu $ is not supported 
on a hyperplane, then $\mu$ admits a log-concave density $f$. Such measures will be called full-dimensional.

Let $f:\mathbb{R}^n\to [0,\infty)$ be a log-concave function with finite, positive integral. 
Its barycenter is defined by
$$\operatorname{bar}(f)= \frac{\int_{\mathbb{R}^n} x\, f(x)\, dx}{\int_{\mathbb{R}^n} f(x)\, dx}.$$
We say that $f$ is centered if $\operatorname{bar}(f)=0$.
We shall use the following result of Fradelizi \cite{Fradelizi-1997}: if $f$ is a centered log-concave density on ${\mathbb R}^n$, then
\begin{equation}\label{eq:frad-2}\|f\|_{\infty }\ls e^nf(0).\end{equation}
The isotropic constant of a log-concave function $f$ with finite positive integral is the affine-invariant quantity
\begin{equation}\label{eq:definition-isotropic}
L_f:= \left( \frac{\|f\|_{\infty}}{\int_{\mathbb{R}^n} f(x)\, dx} \right)^{1/n} \det(\operatorname{Cov}(f))^{1/(2n)},
\end{equation}
where $\operatorname{Cov}(f)$ denotes the covariance matrix of $f$ with entries
\begin{equation*}
\textrm{Cov}(f)_{i,j}:=\frac{\int_{\mathbb{R}^n}x_ix_j f(x)\,dx}{\int_{\mathbb{R}^n} f(x)\,dx}-\frac{\int_{\mathbb{R}^n}x_i f(x)\,dx}{\int_{\mathbb{R}^n} f(x)\,dx}\frac{\int_{\mathbb{R}^n}x_j f(x)\,dx}{\int_{\mathbb{R}^n} f(x)\,dx}.
\end{equation*}
A log-concave function $f$ is called isotropic if
$$\operatorname{bar}(f)=0,\quad \int_{\mathbb{R}^n}f(x)\,dx=1,\quad \text{and} \quad \operatorname{Cov}(f)=I_n.$$
In this case, $L_f=\|f\|_{\infty}^{1/n}$. A full-dimensional log-concave probability measure $\mu$ on $\mathbb{R}^n$ is called isotropic if its
density $f$ is isotropic. Then, we set $L_{\mu}:=L_f$.

Note that a centered convex body $K$ in $\mathbb{R}^n$ with $\vol_n(K)=1$ is isotropic if and only if the log-concave function $L_K^n\mathds{1}_{K/L_K}$ is isotropic.

Let $\mu$ be a full-dimensional log-concave probability measure on $\mathbb{R}^n$.
For any $1\ls k \ls n-1$ and any $k$-dimensional subspace $F\subset\mathbb{R}^n$, the marginal of $\mu$ onto $F$ is defined by
$$\pi_F(\mu)(B):=\mu(P_F^{-1}(B)),$$
for every Borel set $B\subset F$. The measure $\pi_F(\mu)$ is log-concave and admits a density
$$(\pi_F f)(x)=\int_{x+F^\perp} f(y)\,dy.$$
If $f$ is centered (respectively isotropic), then so is $\pi_F f$
(see~\cite[Proposition~5.1.11]{BGVV-book}).

In particular, if $\mu$ is isotropic and $F_{\xi}=\{t\xi:t\in\mathbb{R}\}$ for $\xi\in S^{n-1}$, 
then the one-dimensional marginal $$g_{\xi}(t)=(\pi_{F_{\xi}}(f))(t)=\int_{t\xi +F_{\xi}^{\perp}}f(x)\,dx$$
is an isotropic log-concave density on $\mathbb{R}$. Consequently,
\begin{equation}\label{eq:dim-1}\|g_{\xi}\|_{\infty}=L_{g_{\xi}}\ls 1.\end{equation}
More generally, it was shown in \cite{Fradelizi-1997} that ${\rm Var}(X)\|g\|_{\infty}^2\ls 1$ for any 
centered log-concave density $g$ on $\mathbb{R}$ and $X\sim g$. Bobkov and Chistyakov \cite[Proposition~2.1]{Bobkov-Chistyakov-2015} proved that
$$\frac{1}{12}\ls {\rm Var}(X)\|g\|_{\infty}^2\ls 1$$
for every log-concave density $g$ on $\mathbb{R}$. Moreover, the left inequality holds without the log-concavity
assumption. If $g$ is even, then
$\|g\|_{\infty}=g(0)$ and
$${\rm Var}(X)g(0)^2\ls\frac{1}{2}.$$
In particular, if $g$ is even and isotropic, then ${\rm Var}(X)=1$, and hence 
\begin{equation}\label{eq:dim-12}\|g\|_{\infty}=g(0)\ls 1/\sqrt{2}.\end{equation}
This result was first obtained by Hensley \cite{Hensley-1980}; the symmetry assumption was further relaxed in \cite{Bobkov-1999}.

\smallskip

It is known that every centered log-concave density $f$ admits an isotropic position: there exists $T\in GL_n$ such that the push-forward density
$$f_T(x)=\frac{1}{|\det T|}\,f(T^{-1}x)$$
is isotropic (see \cite[Section~2.3]{BGVV-book}).
Moreover, $f_T$ is also log-concave, and $L_{f_T}=L_f$. It is also known (see \cite[Proposition~2.3.12]{BGVV-book}) that $L_f\gr c$ for every
isotropic log-concave function $f$ on $\mathbb{R}^n$, where $c>0$ is an absolute constant. On the other hand, Ball~\cite{Ball-1988} proved that for every $n$,
$$\tilde{L}_n:=\sup\big\{L_f: f\;\text{is a log-concave density on $\mathbb{R}^n$}\big\}\ls C_1L_n,$$
and hence $\tilde{L}_n\ls C_2$ by the affirmative solution of Bourgain's slicing problem. 

We refer to~\cite{AGA-book,AGA-book-2} for asymptotic convex geometry, and to~\cite{BGVV-book} for background 
on isotropic convex bodies and log-concave measures.

\section{Super-level sets of an isotropic log-concave density}\label{section:3}

Let $\mu$ be an isotropic log-concave probability measure on $\mathbb{R}^n$ with density $f$. For every $t\gr 0$, we define the super-level set
$$R_t(\mu)=\{x\in\mathbb{R}^n:\, f(x)\gr e^{-t}\|f\|_{\infty}\}.$$
Since $f$ is log-concave, each $R_t(\mu)$ is convex, and for $t>0$ we have $0\in\mathrm{int}(R_t(\mu))$.
Moreover, using the fact that every log-concave function with finite positive integral satisfies (see~\cite[Lemma~2.2.1]{BGVV-book})
\begin{equation}\label{eq:A-B}
f(x)\ls A e^{-B|x|}
\qquad\text{for all } x\in\mathbb{R}^n,
\end{equation}
for some constants $A,B>0$, we see that if $x\in R_t(\mu)$ then
$$|x|\ls \frac{1}{B}\bigl(\ln(A/\|f\|_{\infty})+t\bigr).$$
In particular, $R_t(\mu)$ is bounded. The family of bodies $\{R_t(\mu)\}_{t>0}$ is clearly increasing, and the next lemma shows
that $\mu(R_t(\mu))$ converges rapidly to $1$ as $t\to \infty$.

\begin{lemma}\label{lem:rapid}Let $\mu$ be a centered log-concave probability measure on $\mathbb{R}^n$.
Then, for every $t\gr 6n$, $\mu(R_t(\mu ))\gr 1-e^{-t/5}$.
\end{lemma}

\begin{proof}Define $S_t(\mu)=\{x\in\mathbb{R}^n:f(x)\gr e^{-t}f(0)\}$. It is shown in \cite[Proposition~2.1]{Giannopoulos-Tziotziou-2025} that
if $r\gr 5(n-1)$, then $\mu(S_r(\mu))\gr 1-e^{-r/4}$. Since $\|f\|_{\infty}\ls e^nf(0)$ by~\eqref{eq:frad-2}, we have 
$$e^{-t}\|f\|_{\infty}\ls e^{-t+n}f(0)$$
and therefore $S_{t-n}(\mu)\subseteq R_t(\mu)$ for all $t>n$. In particular, if $t\gr 6n$, then $r=t-n\gr 5n$, and hence 
$$\mu(R_t(\mu))\gr 1-e^{-(t-n)/4}\gr 1-e^{-t/5},$$
which proves the claim.
\end{proof}

The proof of the next lemma is essentially contained in \cite[Lemma~5.4]{Klartag-2007}. 

\begin{lemma}\label{lem:inradius}
Let $\mu$ be an isotropic log-concave probability measure on $\mathbb{R}^n$, with $n\gr 3$. Then
$$r(R_{6n}(\mu))\gr 1/3.$$
\end{lemma}

\begin{proof}Let $r:=r(R_{6n}(\mu))$ denote the inradius of $R_{6n}(\mu)$ with respect to the origin. Then there exists $\xi\in S^{n-1}$ such that
$h_{R_{6n}(\mu)}(\xi)\ls r$, and consequently
$$R_{6n}(\mu)\subseteq \{x\in\mathbb{R}^n:\langle x,\xi\rangle \ls r\}.$$
We decompose this half-space as
$$\{x\in\mathbb{R}^n:\langle x,\xi\rangle \ls r\}=\{x\in\mathbb{R}^n:\langle x,\xi\rangle \ls 0\}\cup \{x\in\mathbb{R}^n:0<\langle x,\xi\rangle \ls r\}.$$
Since $\mu$ is centered, Gr\"{u}nbaum's lemma (see~\cite[Lemma~2.2.6]{BGVV-book}) yields 
$$\mu(\{x\in\mathbb{R}^n:\langle x,\xi\rangle\ls 0\})\ls 1-\frac{1}{e}.$$
Let $F_{\xi}=\{t\xi:t\in\mathbb{R}\}$ be the one-dimensional subspace spanned by $\xi$. From \eqref{eq:dim-1},
the density $g_{\xi}$ of the marginal $\pi_{F_{\xi}}(\mu)$ satisfies
\begin{equation*}\|g_{\xi}\|_{\infty}\ls 1.\end{equation*}
therefore,
$$\mu(\{x\in\mathbb{R}^n:0<\langle x,\xi\rangle\ls r\})=\int_0^r\left(\int_{t\xi +F_{\xi}^{\perp}}f(x)\,dx\right)\,dt=\int_0^rg_{\xi}(t)\,dt\ls r.$$
Combining these estimates and using Lemma~\ref{lem:rapid}, we obtain
$$1-e^{-6n/5}\ls \mu(R_{6n}(\mu))\ls 1-\frac{1}{e}+r.$$
For $n\gr 3$, this implies 
$$r\gr e^{-1}-e^{-6n/5}\gr e^{-1}-e^{-18/5}\gr 1/3,$$
as claimed.\end{proof}

\begin{remark}\label{rem:18}\rm We shall use the following observation. Let $\mu$ be a centered log-concave probability measure 
on $\mathbb{R}^n$ with density $f$ satisfying $f(0)=\|f\|_{\infty}$. For any $x\in R_t(\mu)$ and any $\tau\in (0,1)$, 
\begin{equation}\label{eq:scaling}f(\tau x)\gr f(x)^{\tau}f(0)^{1-\tau}\gr e^{-\tau t}f(0)^{\tau}f(0)^{1-\tau}=e^{-\tau t}f(0).\end{equation}
This inequality relies only on log-concavity and the normalization $f(0)=\|f\|_{\infty}$. As a consequence,
\begin{equation}\label{eq:R-concavity}\tau R_t(\mu)\subseteq R_{\tau t}(\mu),\qquad t>0,\;\tau\in (0,1).\end{equation}
In particular, if $\mu$ is isotropic, then 
$$r(R_n(\mu))\gr r(R_{6n}(\mu))/6\gr 1/18.$$
\end{remark}

\section{Upper bound for the perimeter}\label{section:4}

In this section we prove the general upper bounds for $\Gamma_n$ and $\Gamma_n^{(s)}$. Throughout, $\mu$ denotes 
an isotropic log-concave probability measure on $\mathbb{R}^n$ with density $f$.

\smallskip 

We begin with a simple estimate for the measure of a convex body in terms of its
minimal width. Recall that for a convex body $A\subset\mathbb{R}^n$,
$w_A=\min\{w_A(\xi):\xi\in S^{n-1}\}$, where $w_A(\xi)=h_A(\xi)+h_A(-\xi)$
is the width of $A$ in the direction of $\xi\in S^{n-1}$.

\begin{lemma}\label{lem:measure-width}Let $\mu$ be an isotropic log-concave probability measure on $\mathbb{R}^n$.
Then for every convex body $A\subset\mathbb{R}^n$,
$$\mu(A)\ls w_A.$$
\end{lemma}

\begin{proof}Fix $\xi\in S^{n-1}$ and let $F_{\xi}=\{t\xi:t\in\mathbb{R}\}$.
By \eqref{eq:dim-1}, the density $g_{\xi}$ of the marginal $\pi_{F_{\xi}}(\mu)$ satisfies
$\|g_{\xi}\|_{\infty}\ls 1$.

Since 
$$A\subseteq S_{\xi}=\{x\in\mathbb{R}^n:-h_A(-\xi)\ls\langle x,\xi\rangle \ls h_A(\xi)\},$$
we obtain
$$\mu(A)\ls\mu(S_{\xi})=\int_{-h_A(-\xi)}^{h_A(\xi)}\int_{t\xi +F_{\xi}^{\perp}}f(x)\,dx\,dt=\int_{-h_A(-\xi)}^{h_A(\xi)}g_{\xi}(t)\,dt
\ls w_A(\xi).$$
Taking the minimum over $\xi\in S^{n-1}$ yields the result.\end{proof}

\smallskip 

The next lemma is a standard dilation estimate for log-concave measures (see~\cite[Lemma~3.1]{Giannopoulos-Tziotziou-2025}). 
Actually, we shall prove and use a more elaborate version later on, so we skip the details.

\begin{lemma}\label{lem:dilation}Let $\mu$ be a centered log-concave probability measure on ${\mathbb R}^n$. Then for
every $\delta>0$ and every Borel set $A\subset\mathbb{R}^n$,
$$\mu((1+\delta)A)\ls e^{2n\delta }\mu(A).$$
\end{lemma}

\begin{note*}\rm If the density of $\mu$ is a geometric log-concave function, i.e. $f(0)=\|f\|_{\infty}$, 
then the above estimate improves to $\mu((1+\delta)A)\ls e^{n\delta }\mu(A).$
\end{note*}

Using Lemma~\ref{lem:dilation} we can control the perimeter of a convex body containing the origin in terms of its measure
and inradius with respect to the origin.

\begin{lemma}\label{lem:center-origin}
Let $\mu$ be a centered log-concave probability measure on $\mathbb{R}^n$, and let $A\subset\mathbb{R}^n$ be a convex body 
such that $rB_2^n\subseteq A$ for some $r>0$. Then
$$\mu^+(\partial A)\ls \frac{2n\mu(A)}{r}.$$
\end{lemma}

\begin{proof}
Since $A$ is convex, we have $tA+sA=(t+s)A$ for any $t,s>0$. Using also Lemma~\ref{lem:dilation}, we write
$$\mu(A+\epsilon B_2^n)\ls \mu(A+(\epsilon /r)A)= \mu((1+\epsilon /r)A)\ls e^{2n\epsilon /r}\mu(A).$$
Hence,
$$\mu^+(\partial A) =\liminf_{\epsilon\to 0}\frac{\mu(A+\epsilon B_2^n)-\mu(A)}{\epsilon}\ls
\lim_{\epsilon\to 0}\frac{e^{2n\epsilon /r}-1}{\epsilon }\mu(A)=\frac{2n \mu(A)}{r},$$
as claimed.
\end{proof} 

The last fact that we need is the simple observation that if $A$ is a symmetric convex body in $\mathbb{R}^n$ then there exists an inball of $A$ centered  
at the origin, and 
\begin{equation}\label{eq:symmetric-w-r}
w_A=2r_A,
\end{equation} 
where $w_A$ is the minimal width and $r_A$ is the inradius of $A$.

We can now give a one-line proof of Theorem~\ref{th:symmetric-main-upper-bound}.

\begin{proof}[Proof of Theorem~$\ref{th:symmetric-main-upper-bound}$]
Let $A$ be symmetric. Then $r_AB_2^n\subseteq A$ and $w_A=2r_A$. By Lemmas~\ref{lem:center-origin} and \ref{lem:measure-width},
$$\mu^+(\partial A)\ls \frac{2n\mu(A)}{r_A}\ls \frac{2n w_A}{r_A}=4n,$$
and the theorem follows.
\end{proof}

\begin{note*}\rm If the density of $\mu$ is geometric log-concave, then the improved dilation estimate yields 
$\mu^+(\partial A)\ls 2n$ for every symmetric convex body $A\subset\mathbb{R}^n$.
\end{note*}

\smallskip 

We now turn to the proof of Theorem~\ref{th:main-upper-bound}. 
In the general case, $A$ need not be symmetric, and also the origin may not be an interior point of $A$. Recall that, given a convex body 
$A \subset \mathbb{R}^n$, we denote by $r_A$ the radius of the largest Euclidean ball contained in $A$ and by $x_A$ the center of such a ball. 
In other words, $x_A+r_AB_2^n$ is an inball of $A$.

\begin{lemma}\label{lem:center-general-1}Let $\mu$ be a centered log-concave probability measure on $\mathbb{R}^n$. Let $A$ be a convex body in $\mathbb{R}^n$
and let $x_A+r_AB_2^n\subseteq A$ be an inball of $A$. Then
$$\mu^+(\partial A)\ls \frac{n+\ln(\|f\big|_A\|_{\infty}/f(x_A))}{r_A}\mu(A).$$
\end{lemma}

\begin{proof}Let $K=-x_A+A$, so that $0\in K$ and $r_AB_2^n\subseteq K$. For $\epsilon>0$, we have
$$A+\epsilon B_2^n=x_A+K+\epsilon B_2^n\subseteq x_A+K+(\epsilon /r_A)K=x_A+(1+\epsilon /r_A)K.$$
Letting $\delta =\epsilon /r_A$,
$$\mu(A+\epsilon B_2^n)\ls \int_{x_A+(1+\delta)K}f(u)\,du=\int_{(1+\delta)K}f(x_A+v)\,dv=(1+\delta)^n\int_Kf(x_A+(1+\delta)y)\,dy.$$ 
Writing
$$x_A+y=\frac{1}{1+\delta}(x_A+(1+\delta)y)+\frac{\delta}{1+\delta}x_A$$
and using log-concavity of $f$, we get $f(x_A+y)^{1+\delta}\gr f(x_A+(1+\delta)y)f(x_A)^{\delta}$, and hence
\begin{equation}\label{eq:important}f(x_A+(1+\delta)y)\ls f(x_A+y)\left(\frac{f(x_A+y)}{f(x_A)}\right)^{\delta}
\ls f(x_A+y)\left(\frac{\|f\big|_A\|_{\infty}}{f(x_A)}\right)^{\delta}.\end{equation}
Combining the above, we see that
$$\mu(A+\epsilon B_2^n)\ls (1+\delta)^n\left(\frac{\|f\big|_A\|_{\infty}}{f(x_A)}\right)^{\delta}\int_Kf(x_A+y)\,dy
\ls \exp\left(\delta\left(n+\ln\left(\|f\big|_A\|_{\infty}/f(x_A)\right)\right)\right)\mu(A).$$
It follows that
\begin{align*}\mu^+(\partial A) 
&=\liminf_{\epsilon\to 0^+}\frac{\mu(A+\epsilon B_2^n)-\mu(A)}{\epsilon }
\ls \mu(A)\lim_{\epsilon\to 0^+}
\frac{\exp\big(\epsilon /r_A(n+\ln(\|f\big|_A\|_{\infty}/f(x_A)))\big)-1}{\epsilon }\\
&=\mu(A)\frac{n+\ln(\|f\big|_A\|_{\infty}/f(x_A))}{r_A},
\end{align*}
as claimed. \end{proof} 

\begin{remark}\rm The argument above gives in fact a stronger estimate. Let $\mu_A$ be the probability measure with density $(\mu(A))^{-1}f(y)\mathds{1}_A(y)$. 
From \eqref{eq:important} we see that
\begin{align*}
\mu(A+\epsilon B_2^n) &\ls (1+\delta)^n\int_Kf(x_A+(1+\delta)y)\,dy
\ls e^{n\delta}\int_Kf(x_A+y)\left(\frac{f(x_A+y)}{f(x_A)}\right)^{\delta}\,dy\\
&=e^{n\delta}\mu(A)\int_A\left(\frac{f(y)}{f(x_A)}\right)^{\delta}d\mu_A(y).
\end{align*}
It follows that
\begin{align*}\mu^+(\partial A) &=\liminf_{\epsilon\to 0^+}\frac{\mu(A+\epsilon B_2^n)-\mu(A)}{\epsilon }\ls \mu(A)\lim_{\epsilon\to 0^+}
\frac{1}{\epsilon }\left[e^{n\epsilon /r_A}\left(\int_A\left(\frac{f(y)}{f(x_A)}\right)^{\epsilon/r_A}d\mu_A(y)\right)-1\right]\\
&=\frac{\mu(A)}{r_A}\left[n+\int_A\ln \left(\frac{f(y)}{f(x_A)}\right)\,d\mu_A(y)\right]. 
\end{align*}
Using Lemma~\ref{lem:measure-width} we obtain
$$\mu^+(\partial A)\ls \frac{w_A}{r_A}\left[n+\int_A\ln \left(\frac{f(y)}{f(x_A)}\right)\,d\mu_A(y)\right].$$
\end{remark}

For the proof of Theorem~\ref{th:main-upper-bound} we shall use a theorem of Steinhagen \cite{Steinhagen-1922} (see also \cite[Theorem~1.5.2]{Toth-book}) 
to compare the minimal width $w_A$ of a convex body $A$ in $\mathbb{R}^n$ with its inradius $r_A$.

\begin{theorem}[Steinhagen]\label{th:steinhagen}For every convex body $A\subset\mathbb{R}^n$,
$$w_A\ls 2\sqrt{n}r_A.$$
\end{theorem}

As a first step, combining Theorem~\ref{th:steinhagen} with Lemmas~\ref{lem:center-general-1} and~\ref{lem:measure-width} we obtain the following estimate.

\begin{proposition}\label{prop:perimeter-general-1}Let $\mu$ be an isotropic log-concave probability measure on $\mathbb{R}^n$. Then for every  
convex body $A\subset\mathbb{R}^n$, 
$$\mu^+(\partial A)\ls 2\sqrt{n}\big(n+\ln(\|f\big|_A\|_{\infty}/f(x_A))\big),$$
where $x_A$ is the center of the ball of radius $r_A$ inscribed in $A$.
\end{proposition}

\begin{proof}From Lemma~\ref{lem:center-general-1} and Lemma~\ref{lem:measure-width} we have
$$\mu^+(\partial A)\ls \mu(A)\frac{n+\ln(\|f\big|_A\|_{\infty}/f(x_A))}{r_A}\ls \big(n+\ln(\|f\big|_A\|_{\infty}/f(x_A))\big)\,\frac{w_A}{r_A}.$$
The claim follows from Theorem~\ref{th:steinhagen}.
\end{proof}

\begin{corollary}\label{cor:perimeter-subset-Rn}Let $\mu$ be an isotropic log-concave probability measure on $\mathbb{R}^n$. For every  
convex body $A\subset\mathbb{R}^n$ with $A\subseteq R_{6n}(\mu)$, 
$$\mu^+(\partial A)\ls 14n^{3/2}.$$
\end{corollary}

\begin{proof}From Proposition~\ref{prop:perimeter-general-1} we have that 
$$\mu^+(\partial A)\ls 2\sqrt{n}\big(n+\ln(\|f\big|_A\|_{\infty}/f(x_A))\big),$$
where $x_A$ is the center of the ball of radius $r_A$ inscribed in $A$. Since $x_A\in R_{6n}(\mu)$ we also have that $\|f\big|_A\|_{\infty}/f(x_A)\ls e^{6n}$. Therefore,
$$\mu^+(\partial A)\ls 2\sqrt{n}\big(n+6n\big)=14n\sqrt{n}$$
as claimed.
\end{proof}

For the general case, where $A$ may not be contained in $R_{6n}(\mu)$, and in particular $f(x_A)$ may be arbitrarily small, we combine our approach with some
of the ideas developed by Livshyts in \cite{Livshyts-2021}.

\begin{theorem}\label{th:perimeter}Let $\mu$ be an isotropic log-concave probability measure on $\mathbb{R}^n$. For every  
convex body $A\subset\mathbb{R}^n$, 
$$\mu^+(\partial A)\ls Cn^{3/2},$$
where $C>0$ is an absolute constant.
\end{theorem}

\begin{proof}
Fix an arbitrary convex body $A\subset\mathbb{R}^n$. Note that if $A\subseteq R_{6n}(\mu)$ then the result follows from
Corollary~\ref{cor:perimeter-subset-Rn}. Otherwise, $A\setminus R_{6n}(\mu)\neq\varnothing$.  
 
We decompose $\partial A$ as follows:
$$\partial A=(\partial A\cap R_{6n}(\mu))\cup (\partial A\setminus R_{6n}(\mu)).$$
Note that $\partial A\cap R_{6n}(\mu)\subseteq \partial(A\cap R_{6n}(\mu))$, therefore 
\begin{equation}\label{eq:final-00}\mu^+(\partial A\cap R_{6n}(\mu))\ls \mu^+(\partial(A\cap R_{6n}(\mu)))\ls 14n^{3/2}\end{equation}
by Corollary~\ref{cor:perimeter-subset-Rn} applied to the convex body $A\cap R_{6n}(\mu)\subseteq R_{6n}(\mu)$.

Next, we estimate $\mu^+(\partial A\setminus R_{6n}(\mu))$. Note that
\begin{align}\label{eq:ell1-norm0}\|f\|_1 &=\int_0^{\|f\|_{\infty}}\vol_n(\{x\in\mathbb{R}^n:f(x)\gr t\})\,dt\\
\nonumber &=\int_0^{\infty}e^{-s}\|f\|_{\infty}\vol_n(\{x\in\mathbb{R}^n:f(x)\gr e^{-s}\|f\|_{\infty}\})\,ds\\
\nonumber &=\int_0^{\infty}e^{-s}\|f\|_{\infty}\vol_n(R_s(\mu))\,ds.\end{align}
We write
\begin{align*}
\mu^+(\partial A\setminus R_{6n}(\mu)) &=\int_{\partial A\setminus R_{6n}(\mu)}f(x)\,d\lambda(x)\\
&=\int_{\partial A\setminus R_{6n}(\mu)}\int_0^{\infty}\mathds{1}_{\{(x,t):f(x)\gr t\}}(x,t)\,dt\,d\lambda(x)\\
&=\int_0^{e^{-6n}\|f\|_{\infty}}\lambda (\partial A\cap \{x:e^{-6n}\|f\|_{\infty}\gr f(x)\gr t\})\,dt\\
&\ls\int_0^{e^{-6n}\|f\|_{\infty}}\lambda (\partial A\cap \{x:f(x)\gr t\})\,dt\\
&=\int_{6n}^{\infty}e^{-s}\|f\|_{\infty}\lambda (\partial A\cap R_s(\mu))\,ds,
\end{align*}
making the change of variables $t=e^{-s}\|f\|_{\infty}$. Since $\partial A\cap R_s(\mu)\subseteq \partial (A\cap R_s(\mu))$, and also $A\cap R_s(\mu)\subseteq R_s(\mu)$
and both sets are convex, using the monotonicity of 
Lebesgue surface area of convex bodies with respect to inclusion,
as well as \eqref{eq:surface-inradius} and Lemma~\ref{lem:inradius}, we get
$$\lambda (\partial A\cap R_s(\mu))\ls  \lambda (A\cap R_s(\mu))=S(A\cap R_s(\mu))\ls S(R_s(\mu))\ls \frac{n \vol_n(R_s(\mu))}{r(R_s(\mu))}
\ls 3n\vol_n(R_s(\mu)).$$
It follows that
\begin{equation}\label{eq:final-11}
\mu^+(\partial A\setminus R_{6n}(\mu))\ls 3n\int_{6n}^{\infty}e^{-s}\|f\|_{\infty}\vol_n(R_s(\mu))\,ds
\ls 3n\int_0^{\infty}e^{-s}\|f\|_{\infty}\vol_n(R_s(\mu))\,ds=3n\|f\|_1\end{equation}
using \eqref{eq:ell1-norm0} in the end. Combining \eqref{eq:final-11} with \eqref{eq:final-00} we obtain
$$\mu^+(\partial A)\ls 14n^{3/2}+3n\,\|f\|_1=(14+3/\sqrt{n})n^{3/2}.$$
This completes the proof of the theorem. \end{proof}

\section{Improved bounds under additional structure}\label{section:5}

In this section we show that the estimate $O(n^{3/2})$ in Theorem~\ref{th:main-upper-bound} can be improved under additional geometric
assumptions on the measure. More precisely, we establish linear bounds of order $O(n)$ for several natural classes of isotropic log-concave measures.

Besides the results of the previous section, we shall exploit a general upper estimate for $\Gamma(\mu)$ due to Livshyts \cite{Livshyts-2021}. 
For any integrable function $f$ whose level sets $R_t(\mu)$ are convex, Livshyts established the following result.

\begin{theorem}[Livshyts]\label{th:livshyts-1}
Let $f:\mathbb{R}^n\to [0,\infty)$ be an integrable function such that $R_t(\mu)$ is convex for all $t\gr 0$. 
Then
$$\Gamma(\mu)\ls\inf\left\{n\,\frac{\|f\|_{\infty}\,\vol_n(R_t(\mu))+\|f\|_1}{r(R_t(\mu))}:t>0\right\}$$
where $\mu$ is the measure on $\mathbb{R}^n$ with density $f$.
\end{theorem}

For isotropic log-concave functions $f$, we have $\|f\|_{\infty} = L_f^n$ and $\|f\|_1 = 1$.
Thus, to obtain an upper bound for $\Gamma(\mu)$, it suffices to choose $t>0$ such that $\vol_n(R_t(\mu))$ is suitably bounded from above 
and $r(R_t(\mu))$ is suitably bounded from below. 
In the isotropic log-concave case, Livshyts' method yields $\Gamma(\mu) = O(n^2)$ for an appropriate choice of $t$.

\bigskip

\noindent {\bf \S 5.1. Uniform measures on convex bodies.} 
We begin with the case of uniform measures on convex bodies. 

\begin{proposition}\label{prop:main-estimate-body-case}
Let $K$ be an isotropic convex body in $\mathbb{R}^n$ and let $\mu_K$ be the corresponding isotropic probability measure
with density $f_K:=L_K^n\mathds{1}_{K/L_K}$. Then
\begin{equation}\label{eq:uniform}
\Gamma(\mu_K)=L_KS(K).
\end{equation}
In particular, $\Gamma(\mu_K)\ls \sqrt{\frac{n}{n+2}}\,n$.
\end{proposition}

\begin{proof}For any convex body $A$ in $\mathbb{R}^n$ we may write
\begin{align*}\mu_K^+(\partial A) &=\int_{\partial A\cap L_K^{-1}K}f_K(x)\,d\lambda(x)\ls \int_{\partial (A\cap L_K^{-1}K)}f_K(x)\,d\lambda(x)\\
&=L_K^nS(A\cap L_K^{-1}K)\ls L_K^nS(L_K^{-1}K)=L_KS(K),\end{align*}
with equality if $A=L_K^{-1}K$. It follows that
$$\Gamma(\mu_K)=L_KS(K).$$ 
For the second claim, using the general bound $S(K)\ls n\,\vol_n(K)/r(K)$ from \eqref{eq:surface-inradius}, 
together with the normalization $\vol_n(K)=1$, we obtain
$$\Gamma(\mu_K)\ls nL_K/r(K),$$
The lower bound $r(K)\gr\sqrt{\frac{n+2}{n}} L_K$ from~\eqref{eq:KLS-radius} completes the proof. \end{proof}

\begin{remark}\label{rem:optimal}\rm The estimate $\Gamma(\mu_K)\ls \sqrt{\frac{n}{n+2}}\,n$ 
for uniform measures on isotropic convex bodies is sharp. It becomes an equality if $K$ is an isotropic
$n$-dimensional regular simplex. Indeed, let $v_0,\dots,v_n\in\mathbb{R}^n$ satisfy
$$|v_i|=1,\qquad \langle v_i,v_j\rangle=-\frac1n\quad\text{if}\;i\neq j,\qquad \sum_{i=0}^n v_i=0,$$
and consider the regular simplex $\Delta_0:=\mathrm{conv}\{v_0,\dots,v_n\}$, which is centered at the origin.  
Let $\Delta:=\beta\Delta_0$ be the homothetic copy of $\Delta_0$ with $\vol_n(\Delta)=1$, so $\beta^n\vol_n(\Delta_0)=1$. 
We compute $\Gamma(\mu_\Delta)=L_{\Delta}S(\Delta)$.

To compute the inradius, note that for each $i$ the facet $F_i:=\mathrm{conv}\{v_j:\ j\neq i\}$ lies in the hyperplane
$H_i:=\{x:\ \langle x,v_i\rangle=-1/n\}$, because for $j\neq i$ one has $\langle v_j,v_i\rangle=-1/n$. Hence, the distance from
the origin to $H_i$ equals $1/n$, and since the simplex is regular, the origin is the incenter of $\Delta_0$. Therefore,
$$r(\Delta_0)=1/n\qquad\text{and}\qquad r(\Delta)=\beta r(\Delta_0)=\beta/n.$$
To compute the isotropic constant, let $X$ be uniform on $\Delta_0$. Using barycentric coordinates, we may write
$$X=\sum_{i=0}^n \Lambda_i v_i,$$
where $\Lambda=(\Lambda_0,\dots,\Lambda_n)$ is uniform on the simplex $\{\lambda_i\gr 0,\ \sum \lambda_i=1\}$.
Then $\mathbb{E}(\Lambda_i)=\frac1{n+1}$ and
$$\mathbb{E}(\Lambda_i^2)=\frac{2}{(n+1)(n+2)},\qquad \mathbb{E}(\Lambda_i\Lambda_j)=\frac{1}{(n+1)(n+2)}\quad(i\neq j).$$
Since $\sum_{i=0}^n v_i=0$, we have $\mathbb{E}(X)=0$ and
$$\mathrm{Cov}(X)=\mathbb{E}\left(XX^{\mathsf T}\right)=\frac{1}{(n+1)(n+2)}\sum_{i=0}^n v_i v_i^{\mathsf T}.$$
Thus $\Delta$ is isotropic with
$$L_{\Delta}=\frac{\beta}{\sqrt{n(n+2)}}\qquad\text{and hence}\qquad \frac{L_K}{r(K)}=\sqrt{\frac{n}{n+2}}.$$
Finally, for the surface area and $\Gamma(\mu_\Delta)$, note that for any simplex whose incenter is at the origin with inradius $r$, one has the exact identity
$$\vol_n(\Delta)=\frac{r}{n}\,S(\Delta),$$
and since $\vol_n(\Delta)=1$ we obtain $S(\Delta)=\frac{n}{r(\Delta)}$. Therefore,
$$\Gamma(\mu_\Delta)=L_\Delta S(\Delta)=L_\Delta\frac{n}{r(K)}=n\,\frac{L_\Delta}{r(\Delta)}
= \sqrt{\frac{n}{n+2}}\,n.$$
This shows that the estimate in Proposition~\ref{prop:main-estimate-body-case} is optimal.
\end{remark}

\medskip

\noindent {\bf \S 5.2. Measures with homogeneous level sets.}
Let $\mu$ be a log-concave probability measure on $\mathbb{R}^n$ with a geometric log-concave density $f$ satisfying
\begin{equation}\label{eq:homogeneous}
R_t(\mu)=t^{1/p}K \qquad \text{for all } t>0,
\end{equation}
where $K\subset\mathbb{R}^n$ is a fixed convex body containing the origin in its interior and $p\gr 1$.

Applying Theorem~\ref{th:livshyts-1}, we obtain for any $t>0$,
$$\Gamma(\mu)\ls n\,\frac{\|f\|_{\infty}\,\vol_n(R_t(\mu))+1}{r(R_t(\mu))},$$
where we used $\|f\|_1=1$.
We choose $t$ as large as possible under the constraint $\|f\|_{\infty}\vol_n(R_t(\mu))\ls 1$. 
Choosing $t=n/e^p$, we observe that $R_{n/e^p}(\mu)=e^{-1}R_n(\mu)$, and Markov's inequality yields
$$\|f\|_{\infty}\vol_n(R_{n/e^p}(\mu))=\|f\|_{\infty}e^{-n} \vol_n(R_n(\mu)) \ls \mu(R_n(\mu))\ls 1.$$
Since $r(R_{n/e^p}(\mu))=e^{-1}r(R_n(\mu))$ and $r(R_n(\mu))\gr 1/18$ (by Remark~\ref{rem:18}), we conclude that 
$$\Gamma(\mu)\ls \frac{2n}{r(R_{n/e^p}(\mu))}\ls \frac{2en}{r(R_n(\mu))}\ls 36en.$$

\begin{remark}\rm A canonical example is provided by an isotropic measure $\nu_K$ with density
$f_{\nu_K}(x)=c_K e^{-\|x\|_K^p}$, where $K$ is a symmetric convex body and $c_K=(c_{p,n}\vol_n(K))^{-1}$ 
with $c_{p,n}=\int_0^{\infty}pt^{n+p-1}e^{-t^p}dt$. Note that
$$R_n(f_{\nu_K})=n^{1/p}K.$$
Since $\nu_K$ is even and geometric, this implies that $\overline{K}$ has bounded geometric distance from an almost isotropic convex body
in the sense of \cite[Definition~2.5.11]{BGVV-book}: there exists
an almost isotropic convex body $D\subset \mathbb{R}^n$ such that
$$c_1D\subseteq \overline{K}\subseteq c_2D,$$
where $c_1,c_2>0$ are absolute constants (for a proof see \cite[Lemma~3.3]{Giannopoulos-Tziotziou-2025b}).
\end{remark}

\medskip 

\noindent {\bf \S 5.3. $1$-unconditional measures.}
Let $\mu$ be an isotropic log-concave probability measure on $\mathbb{R}^n$ with density $f$ which is $1$-symmetric, meaning that 
$$f(x_1,\ldots,x_n)=f(\epsilon_1 x_{\sigma(1)},\ldots,\epsilon_n x_{\sigma(n)})$$ for all choices of signs $\epsilon_i\in\{-1,1\}$ and all permutations $\sigma$ of $\{1,\ldots,n\}$. 
The symmetry of $f$ implies that each level set $R_t(\mu)$ is a $1$-symmetric convex body.
In particular, $R_t(\mu)$ is in John's position: the maximal-volume ellipsoid contained in $R_t(\mu)$ is the Euclidean ball $r(R_t(\mu))B_2^n$. 
Indeed, the maximal volume ellipsoid of $R_t(\mu)$ is also $1$-symmetric, which forces it to be a centered ball.
K.~Ball's volume ratio theorem~\cite{Ball-1989} then yields
$$\frac{\vol_n(R_t(\mu))^{1/n}}{r(R_t(\mu))}\ls \vol_n(Q_n)^{1/n}=2,$$
where $Q_n=[-1,1]^n$ is the cube with inradius $r(Q_n)=1$.
Consequently,
$$\vol_n(R_t(\mu))^{1/n}\ls 2r(R_t(\mu)) \qquad \text{for all } t>0.$$
This allows us to apply Theorem~\ref{th:livshyts-1} with a suitable choice of $t$ and obtain the bound
$$\Gamma(\mu)\ls Cn$$
for an absolute constant $C>0$, already observed in~\cite{Livshyts-2021}.

We now obtain a similar, linear-in-the-dimension bound for the maximal perimeter of any $1$-unconditional isotropic log-concave
probability measure. Let $\mu$ be an isotropic log-concave probability measure on $\mathbb{R}^n$ with density $f$ which is $1$-unconditional, 
meaning that $$f(x_1,\ldots,x_n)=f(\epsilon_1 x_1,\ldots,\epsilon_n x_n)$$ for all choices of signs $\epsilon_i\in\{-1,1\}$. 
For $k=1,\dots,n$, let $e_k$ be the $k$-th standard basis vector and consider the line segment
$$I_k:= [-e_k,e_k].$$
Recall that $B_2^n\subset B_\infty^n$, and that 
$$B_\infty^n = I_1+\cdots+I_n$$ is the Minkowski sum of the segments $I_k$.
We begin with the following simple telescoping lemma.

\begin{lemma}\label{lem:telescoping}
For any Borel set $A\subset\mathbb{R}^n$ and any $\epsilon>0$, define
$$A_0:=A,\qquad A_k:=A_{k-1}+\epsilon I_k\qquad(k=1,\dots,n).$$
Then $A_n=A+\epsilon B_\infty^n$, and
$$\mu(A+\epsilon B_2^n)-\mu(A)\ls \mu(A+\epsilon B_\infty^n)-\mu(A)=\sum_{k=1}^n \mu\big((A_{k-1}+\epsilon I_k)\setminus A_{k-1}\big).$$
\end{lemma}

\begin{proof}
Since $B_2^n\subset B_\infty^n$ we have $A+\epsilon B_2^n\subset A+\epsilon B_\infty^n$, and hence
$\mu(A+\epsilon B_2^n)-\mu(A)\ls \mu(A+\epsilon B_\infty^n)-\mu(A)$.
Moreover, since $B_\infty^n=I_1+\cdots+I_n$, we have $A+\epsilon B_\infty^n=A_n$ by construction.
Finally, the telescoping identity
\begin{equation}\label{eq:telescope}\mu(A_n)-\mu(A_0)=\sum_{k=1}^n \big(\mu(A_k)-\mu(A_{k-1})\big)=\sum_{k=1}^n \mu\big(A_k\setminus A_{k-1}\big)\end{equation}
gives the claim, since $A_k=A_{k-1}+\epsilon I_k$.
\end{proof}

In what follows, for $k\in\{1,\dots,n\}$ we define the $k$-th marginal density
$$g_k(t):=\int_{\mathbb{R}^{n-1}} f(y,t)\,dy,$$
where we write $x=(y,t)$ with $t=x_k$ and $y\in\mathbb{R}^{n-1}$ collecting the remaining coordinates.

\begin{lemma}\label{lem:fiber}Assume that $f$ is unconditional and log-concave. Then for every convex set
$B\subset\mathbb{R}^n$, every $k\in\{1,\dots,n\}$, and every $\epsilon>0$,
$$\mu\big((B+\epsilon I_k)\setminus B\big)\ls 2\epsilon g_k(0).$$
\end{lemma}

\begin{proof}Fix $k$ and write $x=(y,t)$ with $t=x_k$. For each $y\in\mathbb{R}^{n-1}$, the fiber
$$B_y:=\{t\in\mathbb{R}:\ (y,t)\in B\}$$
is an interval, since $B$ is convex. If $B_y=[a(y),b(y)]$ then $(B+\epsilon I_k)_y=[a(y)-\epsilon,b(y)+\epsilon]$, and hence
$$\big((B+\epsilon I_k)\setminus B\big)_y \subset [a(y)-\epsilon,a(y)]\ \cup\ [b(y),b(y)+\epsilon],$$
with the obvious modifications if one of the endpoints is infinite.
	
By Fubini's theorem,
\begin{align*}
\mu\big((B+\epsilon I_k)\setminus B\big)&=\int_{\mathbb{R}^{n-1}} \int_{\big((B+\epsilon I_k)\setminus B\big)_y} f(y,t)\,dt\,dy \\
&\ls \int_{\mathbb{R}^{n-1}}\left(\int_{a(y)-\epsilon}^{a(y)} f(y,t)\,dt+\int_{b(y)}^{b(y)+\epsilon} f(y,t)\,dt\right)\,dy.
\end{align*}
For fixed $y$, the function $t\mapsto f(y,t)$ is log-concave and even, 
by unconditionality in the $k$-th coordinate. Hence it attains its maximum at $t=0$, so $f(y,t)\ls f(y,0)$ for all $t$.
Therefore,
$$\int_{a(y)-\epsilon}^{a(y)} f(y,t)\,dt+\int_{b(y)}^{b(y)+\epsilon} f(y,t)\,dt \ls 2\epsilon\, f(y,0).$$
Integrating in $y$ yields
$$\mu\big((B+\epsilon I_k)\setminus B\big)\ls 2\epsilon \int_{\mathbb{R}^{n-1}} f(y,0)\,dy = 2\epsilon\, g_k(0),$$
as claimed.
\end{proof}

\begin{theorem}\label{th:unconditional}
Let $\mu$ be an isotropic unconditional log-concave probability measure on $\mathbb{R}^n$
with continuous density $f$. Then for every convex set $A\subset\mathbb{R}^n$ and every $\epsilon>0$,
$$\mu\big((A+\epsilon B_2^n)\setminus A\big)\ls\sqrt{2}n\epsilon.$$
Consequently, $\mu^+(\partial A)\ls\sqrt{2}n$ for every convex set $A\subset\mathbb{R}^n$, 
and hence $\Gamma(\mu)\ls \sqrt{2}n$.
\end{theorem}

\begin{proof}Fix a convex set $A\subset\mathbb{R}^n$ and $\epsilon>0$. Define $A_0,\dots,A_n$ as in Lemma~\ref{lem:telescoping}.
Since $A$ is convex and Minkowski sums preserve convexity, each $A_{k-1}$ is convex.
By Lemmas~\ref{lem:telescoping} and~\ref{lem:fiber},
$$\mu\big((A+\epsilon B_2^n)\setminus A\big)\ls \sum_{k=1}^n \mu\big((A_{k-1}+\epsilon I_k)\setminus A_{k-1}\big)\ls \sum_{k=1}^n 2\epsilon g_k(0)
=2\epsilon\sum_{k=1}^n g_k(0).$$
Each marginal density $g_k$ is log-concave and even on $\mathbb{R}$. 
By isotropicity, ${\rm Var}(X_k)=1$ for $X_k\sim g_k$. Applying \eqref{eq:dim-12} to each $g_k$ yields $g_k(0)\ls 1/\sqrt{2}$, and therefore
$$\mu\big((A+\epsilon B_2^n)\setminus A\big)\ls 2\epsilon n/\sqrt{2}=\sqrt{2}n\epsilon.$$
Dividing by $\epsilon$ and letting $\epsilon\to 0$ completes the proof.
\end{proof}

\medskip 

\noindent {\bf \S 5.4. Sharp one-dimensional bound.} Let $\mu$ be a log-concave probability measure on $\mathbb{R}$ with density $f$.
For an interval $I=[a,b]$ and $\epsilon>0$ we write
$$I+\epsilon B_2^1=[a-\epsilon,b+\epsilon],\qquad B_2^1=[-1,1],$$
and we define the perimeter of $I$ with respect to $\mu$ by
$$\mu^{+}(\partial I)
:=\liminf_{\epsilon\to0^+}\frac{\mu([a-\epsilon,b+\epsilon])-\mu([a,b])}{\epsilon}.$$
The maximal perimeter of $\mu$ is defined as
$$\Gamma(\mu):=\sup\big\{\mu^{+}(\partial I): I\subset\mathbb{R} \ \text{an interval}\big\}.$$
The next proposition provides an exact formula for $\Gamma(\mu)$ in one dimension.

\begin{proposition}\label{prop:Gamma_equals_2M}Let $\mu$ be a log-concave probability measure on $\mathbb{R}$ with density $f$, and set $M:=\|f\|_\infty$. Then
$$\Gamma(\mu)=2M.$$
\end{proposition}

\begin{proof}Fix an interval $I=[a,b]$. Then
$$\mu^+(\partial I)=f(a^-)+f(b^+)\ls 2\|f\|_\infty = 2M,$$
and hence $\Gamma(\mu)\ls 2M$.

Conversely, let $x$ be such that $f(x)$ is arbitrarily close to $M$. For $\delta>0$ sufficiently small, the interval 
$I_\delta=[x-\delta,x+\delta]$ is finite and satisfies
$$\mu^+(\partial I_\delta)=f((x-\delta)^-)+f((x+\delta)^+)\to 2f(x)\quad \text{as}\;\delta\to 0^+.$$
Thus $\Gamma(\mu)\gr 2f(x)$. Choosing $x$ so that $f(x)\to M$ we get $\Gamma(\mu)\gr 2M$.
\end{proof}

Proposition~\ref{prop:Gamma_equals_2M} implies a sharp bound for isotropic log-concave densities on $\mathbb{R}$.

\begin{theorem}\label{thm:1d_isotropic_sharp}
Let $\mu$ be an isotropic log-concave probability measure on $\mathbb{R}$ with density $f$. Then
$$\Gamma(\mu)\ls 2.$$
Moreover, the constant $2$ is sharp: there exist isotropic log-concave probability measures $\mu$ for which 
$\Gamma(\mu)$ is arbitrarily close to $2$. In fact, one-sided exponential distributions achieve equality in 
the underlying extremal inequality.
\end{theorem}

\begin{proof}
Let $M=\|f\|_\infty$. By Proposition~\ref{prop:Gamma_equals_2M}, it suffices to show that $M\ls 1$.
	
Let $X$ be a random variable with distribution $\mu$. For a log-concave density on $\mathbb{R}$ we have  
$${\rm Var}(X)\ls \frac{1}{M^2}.$$
Since $\mu$ is isotropic, ${\rm Var}(X)=1$, and hence $M\ls 1$. Therefore, 
$$\Gamma(\mu)=2M\ls 2.$$
To see that this bound is sharp, let $Y\sim\mathrm{Exp}(1)$ and define $X:=Y-1$. Then $\mathbb{E}(X)=0$ and ${\rm Var}(X)=1$. The
density of $X$ is given by
$$f(x)=e^{-(x+1)}\mathds{1}_{\{x\gr -1\}},$$ so $M=\|f\|_\infty=1$. Hence 
$\Gamma(\mu)=2M=2$, attained as a supremum by intervals shrinking to points where $f$ is close to~$1$.
\end{proof}

\medskip 

\noindent {\bf \S 5.5. $O(n)$ bound for product measures.} Combining the results of the previous two subsections, we now derive a linear-in-the-dimension 
upper bound for the maximal perimeter of log-concave product measures. We begin with the following general statement.

\begin{theorem}\label{thm:product-sharp}
Let $\mu=\mu_1\otimes\cdots\otimes\mu_n$ be a product probability measure on $\mathbb{R}^n$,
where each $\mu_k$ has a density $g_k\in L^\infty(\mathbb{R})$. Then, for every convex set $A\subset\mathbb{R}^n$,
\begin{equation}\label{eq:product-main}\mu^{+}(\partial A)\ \ls\ 2\sum_{k=1}^n \|g_k\|_\infty.\end{equation}
Consequently,
$$\Gamma(\mu)\ \ls\ 2\sum_{k=1}^n \|g_k\|_\infty.$$
Moreover, the constant $2$ in \eqref{eq:product-main} is optimal.
\end{theorem}

\begin{proof}As in the proof of Lemma~\ref{lem:telescoping}, for $k=1,\dots,n$ let $I_k:= [-e_k,e_k]$. Given a  
convex set $A\subset\mathbb{R}^n$ and $\epsilon>0$, define
$$A_0:=A,\qquad A_k:=A_{k-1}+\epsilon I_k\qquad(k=1,\dots,n).$$
By Fubini's theorem, Lemma~\ref{lem:telescoping}, and the product structure of $\mu$, we obtain 
\begin{align*}
\mu(A_k)-\mu(A_{k-1})
&=\int_{\mathbb{R}^{n-1}}\Big(\mu_k\big((A_{k-1})_y+[-\epsilon,\epsilon]\big)-\mu_k\big((A_{k-1})_y\big)\Big)\,d\mu_{\widehat{k}}(y),
\end{align*}
where $\mu_k$ is the $k$-th marginal and $\mu_{\widehat{k}}:=\bigotimes_{i\neq k}\mu_i$.

Let $J\subset\mathbb{R}$ be an interval. The set $J+[-\epsilon,\epsilon]$ differs from $J$ only near its endpoints.
Therefore,
$$\mu_k\big((J+[-\epsilon,\epsilon])\setminus J\big)\ls 2\epsilon\,\|g_k\|_\infty,$$
since there are at most two endpoint neighborhoods, each of length $\epsilon$,
and the density is bounded by $\|g_k\|_\infty$. Hence,
$$\mu(A_k)-\mu(A_{k-1})\ls 2\epsilon\,\|g_k\|_\infty.$$
Summing over $k$ and using~\eqref{eq:telescope}, we obtain 
$$\mu\big((A+\epsilon B_2^n)\setminus A\big)\ls 2\epsilon \sum_{k=1}^n \|g_k\|_\infty.$$
Dividing by $\epsilon$ and letting $\epsilon\to 0^+$ yields \eqref{eq:product-main}.
\end{proof}

\begin{corollary}\label{cor:isotropic-product}
Let $\mu=\mu_1\otimes\cdots\otimes\mu_n$ be a product probability measure on $\mathbb{R}^n$, where each $\mu_k$ is a log-concave probability measure on $\mathbb{R}$ 
with density $g_k$ and ${\rm Var}(X_k)=1$, where $X_k\sim\mu_k$. Then
$$\Gamma(\mu)\ls 2n.$$
\end{corollary}

In view of Theorem~\ref{thm:product-sharp}, the corollary follows immediately from the bound $\|g_k\|_\infty\ls 1$, 
which holds by \eqref{eq:dim-1}.

\bigskip

\noindent {\bf Acknowledgement.} The third and fourth named authors acknowledge support by PhD scholarships from the National Technical University of Athens.

\bigskip

\footnotesize
\bibliographystyle{amsplain}

\bigskip

\thanks{\noindent {\bf Keywords:} Maximal perimeter; Convex bodies; Surface area; Isotropic log-concave probability measures.}

\smallskip

\thanks{\noindent {\bf 2020 MSC:} Primary 52A38; Secondary 52A40, 60D05, 46B06.}

\bigskip

\bigskip 

\medskip 

\noindent \textsc{Silouanos \ Brazitikos}: Department of Mathematics \& Applied Mathematics, University of Crete, Voutes Campus, 70013 Heraklion, Greece.

\smallskip

\noindent \textit{E-mail:} \texttt{silouanb@uoc.gr}

\bigskip

\noindent \textsc{Apostolos \ Giannopoulos}: School of Applied Mathematical and Physical Sciences, National Technical University of Athens, Department of Mathematics, Zografou Campus, GR-157 80, Athens, Greece.

\smallskip

\noindent \textit{E-mail:} \texttt{apgiannop@math.ntua.gr}

\bigskip

\noindent \textsc{Antonios \ Hmadi}: School of Applied Mathematical and Physical Sciences, National Technical University of Athens, Department of Mathematics, Zografou Campus, GR-157 80, Athens, Greece.

\smallskip

\noindent \textit{E-mail:} \texttt{ahmadi@mail.ntua.gr}

\bigskip

\noindent \textsc{Natalia \ Tziotziou}: School of Applied Mathematical and Physical Sciences, National Technical University of Athens, Department of Mathematics, Zografou Campus, GR-157 80, Athens, Greece.

\smallskip

\noindent \textit{E-mail:} \texttt{nataliatz99@gmail.com}

\end{document}